\documentclass[a4paper,twoside,12pt]{article}
\usepackage{amsmath,amsthm, amssymb,color}
\usepackage{multicol}
\usepackage{wrapfig}
\usepackage{graphicx}
\usepackage{graphics}
\usepackage{yfonts}
\usepackage{times}
\usepackage{enumerate}

\theoremstyle{definition}

\theoremstyle{definition}

\theoremstyle{plain}
\newtheorem{thm}{Theorem}[section]
\newtheorem{lemma}[thm]{Lemma}

\newtheorem{remark}[thm]{Remark}

\def\be{\begin{eqnarray*}}
\def\ee{\end{eqnarray*}}

\def\ben{\begin{eqnarray}}
\def\een{\end{eqnarray}}
\def\be{\begin{eqnarray*}}
\def\ee{\end{eqnarray*}}

\def\wt{\widetilde}
\def\wh{\widehat}

\newcommand{\dive}{\operatorname{div}}
\newcommand{\Dive}{\operatorname{Div}}

\numberwithin{equation}{section}

\begin{document}

\baselineskip=17pt


\title{Biharmonic obstacle problem: guaranteed and computable error bounds for approximate solutions}

\author{{\textsc{ Darya E.  Apushkinskaya}}\\
Department of Mathematics, Saarland University, P.O. Box 151150\\
66041 Saarbr{\"u}cken, Germany\\
\textit{E-mail: darya@math.uni-sb.de}\\
Peoples Friendship University of Russia (RUDN University)\\
6 Miklukho-Maklaya St, Moscow, 117198, Russian Federation\\
\\
{\textsc{ Sergey I. Repin}}\\
Steklov Institute of Mathematics at St. Petersburg,\\ Fontanka 27, 191023, St. Petersburg,
Russia\\
\textit{E-mail: repin@pdmi.ras.ru}\\
University of Jyv{\"a}skyl{\"a}, P.O. Box 35  (Agora),\\
FIN-40014, Finland \\
\textit{E-mail:serepin@jyu.fi}}

\date{\today}

\maketitle
\bibliographystyle{alpha}
\newcommand{\tg}{\textup{tg}}
\newcommand{\ep}{\varepsilon}
\newcommand{\osc}[1]{\underset{#1}{\textup{osc}}}

\begin{abstract}
The paper is concerned with a free boundary variational problem generated by the biharmonic operator
and an obstacle. The main goal is to deduce a fully guaranteed upper bound of the difference between the exact
minimizer u and any function (approximation) from the corresponding energy class (which consists of the functions
in $H^2$ satisfying the prescribed boundary conditions and the restrictions stipulated by the obstacle).  For this
purpose we use the duality method of the calculus of variations and general type error identities earlier derived
for a wide class of convex variational problems. By this method, we define 
a combined primal--dual measure of error. It contains four  terms of different nature. Two of them are the norms of the difference between the exact solutions (of the direct and dual variational problems) and corresponding approximations. Two others  are nonlinear measures,  related to approximation of the coincidence set (they vanish if the coincidence set defined by means of the approximate solution coincides with the exact one).  The measure satisfies the error identity, which right hand side depends on approximate solutions only and, therefore,  is fully
computable.  Thus, the identity provides direct estimation of the primal--dual errors.
However, it contains
a certain  restriction on the form of the dual approximation. In the second part of the paper, we present a way to  skip the restriction. As a result, we obtain a fully guaranteed and directly computable error majorant
valid for a wide class of approximations regardless  of the method used for their construction. The estimates
are verified in a series of tests with different approximate solutions.
Some of them are quite close to the exact solution and others are rather coarse
and have coincidence sets that differ much from the exact one. The results show that the estimates are robust and effective in all the cases.


\end{abstract}
\section{Introduction}

Let $\Omega$ be an open, connected, and bounded domain in $\mathbb{R}^d$ with Lipschitz continuous boundary $\partial\Omega$, let ${\boldsymbol \nu}$ be an outward unit normal to $\partial\Omega$, and let $\varphi$ be a given function (obstacle) in $C^2(\overline\Omega)$ such that $\varphi \leqslant 0$ on $\partial\Omega$. Throughout the paper, we use the standard notation for the Lebesgue and Sobolev spaces of functions. By $g_{\oplus}$   we denote $\max\{g,0\}$.

For a given function $f \in L^2(\Omega)$, we consider the following variational Problem ($\mathcal{P}$): minimize the functional
\begin{equation} \label{int}
J(v)=\int\limits_{\Omega}\left(\frac{1}{2}|\Delta v|^2-fv\right)dx
\end{equation}
over the closed convex set
$$
\mathbb{K}=\left\{v\in H^2(\Omega): v\big|_{\partial\Omega}=\frac{\partial v}{\partial\mathbf{n}}\bigg|_{\partial\Omega}=0,\ v \geqslant \varphi\ \text{a.e. in}\ \Omega\right\}.
$$
Here $\varphi$ is a given function (obstacle) such that
$\varphi \in C^2(\overline{\Omega})$ and $\varphi \leqslant 0$ on $\partial\Omega$.

Problem ($\mathcal{P}$) is called the \textit{biharmonic obstacle problem} with obstacle $\varphi$. Such problem has many applications in elasticity theory (frictionless equilibrium contact problems of elastic plates or beams over a rigid obstacle) and in fluid mechanics (incompressible fluid flow  at low Reynolds number in $\mathbb{R}^2$). 
By standard results \cite{LS67, L69} the problem ($\mathcal{P}$) has a unique solution $u$ which satisfies a.e.  in $\Omega$ the following relations
\begin{equation} \label{bi_ineq}
\Delta^2 u \geqslant f, \quad u \geqslant \varphi, \quad (\Delta^2 u -f)\cdot (u-\varphi)=0.
\end{equation}

 In particular, by the well-known works \cite{CF79} and \cite{Fr73}, we have the following a priori regularity 
\begin{equation} \label{reg_u}
 u\in H^3_{loc}(\Omega)\cap W^{2,\infty}_{loc}(\Omega)\qquad \text{and} \qquad \Delta u \in W^{2,\infty}_{loc}(\Omega).
\end{equation}

The domain $\Omega$ is divided in two subdomains $\Omega_0$ and $\Omega_{\varphi}$, where  $u$ has different properties. The equation $\Delta^2u=f$ holds in $\Omega_0$, while in $\Omega_{\varphi}$ the function $u$ coincides with the obstacle ($\Omega_{\varphi}$ is called the coincidence set). The interface between $\Omega_0$ and $\Omega_{\varphi}$ is apriori unknown. Therefore, the problem ($\mathcal{P}$) belongs to the class of free boundary problems

The biharmonic obstacle problems have been actively studied by many authors, starting with the pioneering works of Landau and Lifshitz \cite{LaLi59}, Frehse \cite{Fr71, Fr73}, Cimatti \cite{Cim73}, Stampacchia \cite{St75}, and Br\'{e}zis and Stampacchia \cite{BrSt77}.  We mention also the monographs by Duvaut and Lions \cite{DuLio76}, and by Rodrigues \cite{R87}, where some examples of the problem of bending a plate over an obstacle are considered. Notice that most of the studies of the  fourth order obstacle problems  were mainly focused either on regularity of minimizers  or on the properties of the respective free boundaries (see \cite{Fr71, Fr73, Cim73, CF79, CFT82, Sch86, Al19}).

Approximation methods for the biharmonic obstacle problem
have been developed within the framework of computational methods
for variational inequalities (e.g., see \cite{BSZZ12, GMV84, G84, HHN96, ItK90})
and related optimal control problems \cite{AHL10,ItK00}.
Hence, in principle, it is known how to construct a sequence of approximations
converging to the exact minimizer of this nonlinear variational problem.

In this paper, we are concerned with a different question. Our goal is to deduce a guaranteed and fully computable bound for the distance between the exact solution $u\in \mathbb{K}$ and an approximate solution $v\in \mathbb{K}$ measured in terms of natural energy norm. We apply the same method as
was used in \cite{NeR01} for the derivation of guaranteed error bounds
of the difference between exact solution of the linear biharmonic problem and
any function in the energy admissible class of functions. In \cite{Re03,Re08,RV18,AR18} and some other publications this method was
applied to obstacle problems associated with  elliptic  operators
of the second order.
Below we show that it is also quite efficient for higher order operators
and generates a natural error measure together with fully computable bounds 
of this measure for any function in the respective energy space compared
with the exact minimizer.

\section{Estimates of the distance to the exact solution}


\subsection{General form of the error identity}

Consider the functional spaces 
$$
V:=\big\{w \in H^2(\Omega) \mid w|_{\partial\Omega}=\frac{\partial w}{\partial \mathbf{n}}\big|_{\partial\Omega}=0 \big\}
$$ and $N:=L^2(\Omega, M^{d\times d}_{\rm Sym})$, where $M^{d \times d}_{\rm Sym}$ denotes the space of $d\times d$ symmetric matrices. The  corresponding conjugate (dual) spaces are $V^*=H^{-2}(\Omega)$ and $N^*=N$, respectively.

The functional $J$ can be represented in the form
\begin{equation}
 \label{functional}
J(v)=G(\Lambda v)+F(v),
 \end{equation}
where the operator $\Lambda$ and  functionals $G$ and $F$ are defined as follows:
\begin{gather*}
 \Lambda: V \rightarrow N, \quad \Lambda:=\nabla \nabla; \qquad \quad
 G: N \rightarrow \mathbb{R}, \quad
G(n):=\frac{1}{2} \int\limits_{\Omega} |n|^2dx; \\
F: V \rightarrow \mathbb{R}, \quad F(v):=-\int\limits_{\Omega} fv dx+\chi_{\mathbb K} (v), \qquad \chi_{\mathbb K} (v)=\left\{\begin{aligned}
&0, &&v\in \mathbb{K},\\
&+\infty,  &&v\notin \mathbb{K}.
\end{aligned}
\right.
\end{gather*}
Hence we can use the general theory presented in \cite{R03a} and in the book \cite{NeRe04}.
Further, we denote by $p^*$  the exact solution of the dual variational problem, which is to maximize the functional  (cf. \cite{ET76})
\begin{equation}
I^*(n^*):=-G^*(n^*)-F^*(-\Lambda^* n^*)
\label{dual_problem}
\end{equation}
over the space $N^*$. 
Here, the operator $\Lambda^*$ is defined as follows:
$$
\Lambda^*:N \rightarrow V^*, \qquad \Lambda^*:= \dive\Dive,
$$
while
$G^*:N^* \rightarrow \mathbb{R}$ and $F^*: V^*=H^{-2}(\Omega) \rightarrow \mathbb{R}$ are the Young-Fenchel transforms of $G$ and $F$, respectively. 


In view of the duality relation $J(u)=I^*(p^*)$ and the identities (7.2.13)-(7.2.14) from \cite{NeRe04} we have for an arbitrary $v\in \mathbb{K}$ and $n^* \in N^*$   the following relations:
\begin{align}
J(v)-J(u)&=J(v)-I^*(p^*)=\mathcal{D}_G(\Lambda v, p^*)+\mathcal{D}_F(v, -\Lambda^*p^*),
\label{7-2-13a} \\
I^*(p^*)-I^*(n^*)&=J(u)-I^*(n^*)=\mathcal{D}_G(\Lambda u, n^*)+\mathcal{D}_F(u,-\Lambda^*n^*).
\label{7-2-13b}
\end{align}
Here $\mathcal{D}_G$ and $\mathcal{D}_F$ denote the so-called compound functionals which are determined by the relations
\begin{eqnarray}
&&{\mathcal D}_G(\Lambda v,n^*)=G(\Lambda v)+G^*(n^*)-(\Lambda v,n^*),\\
&&{\mathcal D}_F(v,-\Lambda^*n^*)=F(v)+F^*(n^*)+<\Lambda n^*,v>,
\end{eqnarray}
where $<v^*,v>$ stands for a linear functional coupling the elements $v\in V$ and $v^*\in V^*$. It follows from the definition of a conjugate functional that the compound functional is always non-negative.

Moreover, relations (\ref{7-2-13a})-(\ref{7-2-13b}) directly imply (see \cite{R03a} for more details) for any $v\in \mathbb{K}$ and $n^*\in N^*$ the validity of  the error identity
\begin{multline} \label{222_GEN}
\mathcal{D}_G(\Lambda v, p^*)+\mathcal{D}_F(v, -\Lambda^*p^*)+
\mathcal{D}_G(\Lambda u, n^*)+\mathcal{D}_F(u,-\Lambda^*n^*)\\
=\mathcal{D}_G(\Lambda v, n^*)+\mathcal{D}_F(v, -\Lambda^*n^*).
\end{multline}
The left-hand side of (\ref{222_GEN}) contains four terms that can be considered as deviation measures of the functions $v$ and $n^*$ from $u$ and $p^*$, respectively. The first two terms can be treated as a measure $\boldsymbol{\mu}(v)$ characterizing the error  of approximation $v$, while another two terms can be regarded as a measure $\boldsymbol{\mu^*}(n^*)$ indicating the error of dual approximation $n^*$. The right-hand side of (\ref{222_GEN}) consists of two terms that do not contain unknown exact solutions $u$ and $p^*$. Therefore, this r.h.s. can be calculated explicitly. Moreover, from  (\ref{7-2-13a}) and  (\ref{7-2-13b}) it follows that the r.h.s. of (\ref{222_GEN}) coincides with the so-called duality gap $J(v)-I^*(p^*)$. Notice that the sequences of approximations $\{v_k\}$, $\{n^*_k\}$, constructed with the help of variational methods, should minimize this gap. Therefore, the identity (\ref{222_GEN}) shows that the measures on the l.h.s. of (\ref{222_GEN}) must tend to zero, if the sequences $\{v_k\}$, $\{n^*_k\}$ are constructed correctly and converge to exact solutions. Hence, the measure on the l.h.s. of the error identity is  an adequate characteristic of the quality of approximations.

Identity  (\ref{222_GEN}) holds for any variational problem with the functional of the form (\ref{functional}). We establish its form in terms of the studied problem. It is easy to see that $G^*: N^* \rightarrow \mathbb{R}$ is defined by the equality $G^*(n^*):=\frac{1}{2} \|n^*\|^2$ (hereinafter, $\|\cdot\|$ denotes the norm in the spaces $L_2(\Omega)$ for scalar, vector, and matrix functions). Therefore,
the first  terms on the right hand sides of (\ref{7-2-13a}) and (\ref{7-2-13b}) are computed easily:
\begin{align}
\mathcal{D}_G(\Lambda v, p^*)&=\frac{1}{2}\|\nabla \nabla v-p^*\|^2=\frac{1}{2}\|\nabla \nabla (v-u)\|^2, \label{D_Gp*}\\
\mathcal{D}_G(\Lambda u, n^*)&=\frac{1}{2}\|\nabla \nabla u-n^*\|^2=\frac{1}{2}\| p^*-n^*\|^2. \label{D_Gn^*}
\end{align}

A computation of the last  summands on the right hand sides of (\ref{7-2-13a}) and (\ref{7-2-13b}) requires more work. 

To compute $\langle v^*, v\rangle$, we need the intermediate Hilbert space $\mathcal{V}:=L^2(\Omega)$. It is clear that  $\mathcal{V}$ satisfies the inclusions
$
V \subset \mathcal{V} \subset V^*$.
If $v^* \in \mathcal{V}$ then the  scalar product $\langle v^*, v\rangle$, is identified with scalar product in the space $\mathcal{V}$, i.e.,
$$
\langle v^*,v\rangle=\int\limits_{\Omega}v^*vdx.
$$
Notice that $v \in L^2(\Omega)$, so the above integral is well-defined for any $v^* \in \mathcal{V}$ and $v\in V$. 

In accordance with the definition of the conjugate functional, for $n^*\in N$ we have
\begin{align*}
F^*(-\Lambda^*n^*)&=\sup\limits_{v\in \mathbb{K}} \left\{\langle -\Lambda^*n^*,v\rangle+(f,v)\right\}=
\sup\limits_{v\in \mathbb{K}} \left\{-(n^*,\Lambda v)+(f,v)\right\}\\
&=\sup\limits_{v\in \mathbb{K}} \bigg\{
\int\limits_{\Omega}(fv-n^*:\nabla\nabla v)dx \bigg\}.
\end{align*}

Observe that  a function $n^*$ in  (\ref{7-2-13b}) is in our disposal. Therefore,
without loss of generality, we may restrict our consideration for symmetric $n^*$ satisfying
\begin{equation}
\label{regularitynz}
n^* \in H(\Omega,\dive\Dive):=\left\{ m^*\in N^*: \dive\Dive m^* \in L^2(\Omega)\right\}.
\end{equation}
Taking into account the condition $\dfrac{\partial v}{\partial \mathbf{n}}\bigg|_{\partial\Omega}=0$ and using integration by parts we conclude that
$$
0=\int\limits_{\partial\Omega}\mathbf{n}(n^*\nabla v)ds=
\int\limits_{\Omega}\Dive{(n^*\nabla v)}dx=
\int\limits_{\Omega}\left(\nabla v \cdot \Dive{n^*}+n^*:\nabla\nabla v\right)dx.
$$
Hence
\begin{equation} \label{by_parts}
\int\limits_{\Omega}\left(fv-n^*:\nabla\nabla v\right)dx=\int\limits_{\Omega}\left(fv+\nabla v
\cdot\Dive{n^*}\right)dx.
\end{equation}
Combining (\ref{by_parts}) with  the formula
\begin{equation} \label{2nd_by_parts}
0=\int\limits_{\partial\Omega}( \Dive n^*\cdot \mathbf{n})vds=
\int\limits_\Omega\dive (v\Dive n^*) dx=\int\limits_\Omega
(\Dive n^*\cdot \nabla v+v \dive \Dive n^*)dx
\end{equation}
we arrive at
$$
\int\limits_\Omega (fv-n^*:\nabla\nabla v)dx=\int\limits_\Omega (f-\dive\Dive n^*)v\,dx.
$$

Thus, for $n^*\in H(\Omega, \dive\Dive)$ we have
\begin{equation} \label{F*(n)}
F^*(-\Lambda^*n^*)=\sup\limits_{v\in \mathbb{K}} \left\{
\int\limits_\Omega (f-\dive\Dive n^*)v\,dx \right\}.
\end{equation}

Let $ \wh v\in \mathbb{K}$ be a given function. Then the function $\wh v + w$ with 
$$
w\in V^+(\Omega):=\{w\in V(\Omega) : w \geqslant 0 \ \text{a.e. in}\ \Omega\},
$$
also belongs to $\mathbb{K}$.
It is easy to see that 
$$
F^*(-\Lambda^*n^*)\geq\int\limits_\Omega (f-\dive\Dive n^*)\wh v dx +
\sup\limits_{w\in  V^+(\Omega)} \left\{
\int\limits_\Omega (f-\dive\Dive n^*)w\,dx \right\}.
$$
Therefore, this expression is finite if and only if $n^* \in Q^*_{\ominus}$, where
\begin{equation}
\label{Qzmin}
Q^*_{\ominus}:=\bigg\{m^*\!\in H(\Omega, \dive\Dive)\!\mid \!\int\limits_{\Omega}
(f-\dive\Dive m^*) w dx \leqslant 0 \, \forall w\in V^+(\Omega)\bigg\}.
\end{equation}
The integral condition in the definition of $Q^*_\ominus$ means that $f-\dive\Dive n^*~\leq~ 0$ almost everywhere in $\Omega$. Indeed, suppose that $n^*\in Q^*_{\ominus}$ and $f-\dive\Dive n^*>0$ on some nonzero measure set $\omega$. Then there exists a ball $B(x_0,\rho)\subset\omega$ where this inequality holds. Consider a compact function $w\in V^+(\Omega)$ having support in this ball. The function $w$ is positive inside the ball, and, consequently,
$$
\int\limits_{\Omega}
(f-\dive\Dive n^*) w dx > 0.
$$
We get a contradiction that proves the validity of the statement.

Let $n^*\in Q^*_{\ominus}$. It is clear that
\begin{equation*}
\sup\limits_{v\in \mathbb{K}} \
\int\limits_\Omega (f-\dive\Dive n^*)v\,dx\leq
\int\limits_\Omega (f-\dive\Dive n^*)\varphi\,dx.
\end{equation*}
Moreover, there exists a sequence of functions $v_k\in {\mathbb K}$ such that $v_k\rightarrow \varphi$ in $L^2(\Omega)$.
We conclude that
\begin{equation} \label{F*n2}
F^*(-\Lambda^*n^*)=
\left\{
\begin{aligned}
&\int\limits_{\Omega}  \varphi(f-\dive\Dive n^*) dx, &&\text{if}\ \, n^*\in Q^*_{\ominus},\\
&+\infty, &&\text{otherwise}.
\end{aligned}
\right.
\end{equation}

Thus, the compound functional $\mathcal{D}_F(u, -\Lambda^*n^*)$ is finite if and only if  the condition
\begin{equation}
f-\dive \Dive  n^* \leqslant 0 
\label{cond-N}
\end{equation}
is satisfied almost everywhere in $\Omega$. 

Therefore, for $n^*\in H(\Omega, \dive\Dive)$ satisfying  (\ref{cond-N}), the compound functional $\mathcal{D}_F(u,-\Lambda^*n^*)$  has the form
\begin{equation} \label{D_Fn*}
\mathcal{D}_F(u,-\Lambda^*n^*)=\int\limits_{\Omega}(f-\dive\Dive n^*)(\varphi-u)dx.
\end{equation}

Our next goal is to compute $\mathcal{D}_F(v,-\Lambda^*p^*)$ for $v\in \mathbb{K}$. 
We can not use the previous formula since $p^*$, in general, does not satisfy the condition (\ref{regularitynz}). Indeed, due to (\ref{reg_u}) we only know that $\dive\Dive p^*$ is a square integrable function on the set $\Omega_0$ (where $u>\varphi$ and the equality $\dive\Dive p^*=f$ holds) and on the coincidence set $\Omega_{\varphi}=\{u=\varphi\}=\Omega\setminus \Omega_0$. However, the square integrability of $\dive\Dive p^*$ does not hold in the whole domain $\Omega$. By this reason, we use a different argument. 

Setting $u=v$ in (\ref{7-2-13a})
we have 
$$
\mathcal{D}_F(u,-\Lambda^*p^*)=0.
$$
This gives
$$
F^*(-\Lambda^*p^*)=-F(u)-\langle \Lambda^*p^*,u\rangle=\int\limits_{\Omega} (fu-p^*:\nabla\nabla u)dx.
$$
Now, using the above relation and arguing in the same way as in deriving (\ref{by_parts}), we conclude that									\begin{equation*} 	
\begin{aligned}													
\mathcal{D}_F (v,-\Lambda^* p^*)&=F(v)+F^*(-\Lambda^* p^*)+
\langle\Lambda^*p^*,v\rangle
=\int\limits_{\Omega}\left(f(u-v)-p^*:\nabla\nabla (u-v)\right)dx\\
&=\int\limits_{\Omega} \left(f(u-v)+\nabla (u-v) \cdot \Dive p^*\right)dx.
\end{aligned}		
\end{equation*}

We have
\begin{equation*}
\int\limits_\Omega (f(u-v)+\Dive p^*\cdot\nabla (u-v))dx=\int\limits_{\{u>\varphi\}} (\dots)dx+
\int\limits_{\{u=\varphi\}} (\dots)dx.
\end{equation*}
Let ${{\boldsymbol \nu}_{\scriptscriptstyle\Gamma_u}}$ denote the exterior unit normal to $\partial\Omega_{\varphi}$ and $e:=u-v$. Since
\be
&&\int\limits_{\{u>\varphi\}} (fe+\Dive p^*\cdot\nabla e)dx=
\underbrace{\int\limits_{\{u>\varphi\}} (f-\dive\Dive p^*)e dx}-\int\limits_{\Gamma_u} (\Dive p^*\cdot {{\boldsymbol \nu}_{\scriptscriptstyle\Gamma_u}}) e ds\\
&&\qquad\qquad\qquad\qquad\qquad\qquad\qquad\qquad \quad   =0
\ee
and
\be
&&\int\limits_{\{u=\varphi\}} (fe+\Dive p^*\cdot\nabla e)dx=
\int\limits_{\{u=\varphi\}} (f-\dive\Dive p^*)e dx+\int\limits_{\Gamma_u} (\Dive p^*\cdot {{\boldsymbol \nu}_{\scriptscriptstyle\Gamma_u}}) e ds,
\ee
we obtain
\begin{equation} \label{D_Fp*}
\mathcal{D}_F(v,-\Lambda^*p^*)=\int\limits_{\Gamma_u} \bigg[\Dive p^*\cdot {{\boldsymbol \nu}_{\scriptscriptstyle\Gamma_u}}\bigg] (u-v) ds+
\int\limits_{\{u=\varphi\}} (f-\dive\Dive p^*)(u-v)dx,
\end{equation}
where $\big[\Upsilon\big]=\Upsilon_{\Gamma_u}(\Omega_{\varphi})-\Upsilon_{\Gamma_u}(\Omega_0)$ 
denote the jump of $\Upsilon:=\Dive p^*\cdot {{\boldsymbol \nu}_{\scriptscriptstyle\Gamma_u}}$ across the free boundary $\Gamma_u:=\partial\Omega_{\varphi}$.
Thus, we have to take into consideration an additional integral term arising on the free boundary $\Gamma_u$.
\bigskip

Combining (\ref{7-2-13a})-(\ref{D_Gn^*}) with (\ref{D_Fn*})-(\ref{D_Fp*}), 
we get explicit expressions for the measures in the left-hand side of (\ref{222_GEN}). For any function $v\in \mathbb{K}$ we have
\begin{equation}
\boldsymbol{\mu} (v):=\mathcal{D}_G(\Lambda v, p^*)+\mathcal{D}_F(v, -\Lambda^*p^*)=\frac{1}{2}\|\nabla \nabla (u-v)\|^2+\mu_{\varphi }(v),
\label{error2}
\end{equation}
where 
\begin{equation}
\mu_{\varphi }(v):=\int\limits_{\{u=\varphi \}} (\dive\Dive \nabla \nabla u -f)(v-u)dx 
-\int\limits_{\Gamma_u} \bigg[\Dive \nabla\nabla u\cdot {\boldsymbol \nu}_{\Gamma_u}\bigg] (v-u)ds.
\label{measure1}
\end{equation}
The first term in (\ref{error2}) controls the deviation 
from $u$ in the  $H^2$--norm. The second term $\mu_{\varphi }(v)$ defined by (\ref{measure1}) can be viewed as an additional (nonlinear) measure of  $v-u$. This term is nonegative and vanishes if $v=u$.
Indeed, it is well-known that the problem ($\mathcal{P}$) is equivalent to the biharmonic variational inequality: find $u\in \mathbb{K}$ such that
\begin{equation} \label{var_ineq}
\int\limits_{\Omega_0\cup\Omega_{\varphi}}\left\{\nabla\nabla u:\nabla\nabla (v-u) -f(v-u)\right\}dx \geqslant 0 \qquad \forall v\in \mathbb{K}.
\end{equation}
Applying integration by parts two times and arguing in the same manner as in deriving (\ref{D_Fp*}),  we transform the inequality (\ref{var_ineq}) to the form
$$
\mu_{\varphi}(v) \geqslant 0 \qquad 
\forall v\in \mathbb{K}.
$$
Notice that the first term in $\mu_{\varphi }(v)$ is quite analogous to that in the classical obstacle
problem ((see \cite{RV18}). The expression $f-\dive\Dive p^*$ plays the role of a "weight" function which is negative,
so that the whole integral is zero or positive. 
The second term in $\mu_{\varphi }(v)$ is of a new type. It also serves as a penalty
in the line integral. Notice that if $v=u$ on $\Gamma_u$, then it wanishes.

It is easy to see that $\mu_{\varphi}(v)=0$ if $\Omega_\varphi \subset \{x\in \Omega\,\mid\,v(x)=\varphi(x)\}$. In other cases, this measure will be positive.
Thus, the measure $\mu_{\varphi}(v)$ controls (in a
weak integral sense) how accurately the set $\{v=\varphi\}$ approximate the exact coincidence set $\Omega_{\varphi}$. 
Obviously, this component is very week and does not provide the desired information about the free boundary $\Gamma_u$. We emphasize that it is impossible to get more information about the free boundary in the framework of the standard variational approach. Indeed, in view of the equality $\boldsymbol{\mu} (v)=J(v)-J(u)$, the measure $\boldsymbol{\mu} (v)$ tends to zero at all minimizing sequences. Moreover, this measure is the strongest among all measures that possess such a property.

Similarly, we see that if $p^*$ is the maximizer of the dual variational problem (\ref{dual_problem}) and $n^* \in H(\Omega, \dive\Dive )$ is its approximation satisfying the condition (\ref{cond-N}), then the corresponding deviation measure has the form
\begin{equation}
\boldsymbol{\mu}^*(n^*):=
\mathcal{D}_G(\Lambda u, n^*)+\mathcal{D}_F(u,-\Lambda^*n^*)=\frac{1}{2}\|p^*-n^*\|^2+
\mu^*_{\varphi }(n^*),
\label{error3} 
\end{equation}
where
\begin{equation}
\mu^*_{\varphi }(n^*):=\int\limits_{\Omega_0} (f-\dive\Dive n^*)(\varphi-u)dx.
\label{measure2}
\end{equation}
The integral in (\ref{measure2}) is non-negative. It can be considered 
as a measure penalizing (in weak integral sense) an incorrect behavior of the dual variable  on the set $\Omega_0$ where $n^*$ must satisfy the differential equation. Finally, we note that $\boldsymbol{\mu}^*(n^*)=I^*(p)-I^*(n^*)$. Therefore, $(n^*_k)$ is a maximizing sequence in the dual problem if and only if this error measure tends to zero. 
Regrettably, both  measures $\mu_{\varphi }$ and $\mu^*_{\varphi }$  are too weak to estimate how accurately the free boundary $\Gamma_u$ is reproduced by the approximate solution. This fact shows  limitations of direct variational methods in reconstruction of free boundaries (see also  \cite{RV18}).
\medskip

The equalities  (\ref{222_GEN}), (\ref{F*(n)}), (\ref{D_Fn*}), (\ref{error2}),  and  (\ref{error3}) imply the following result:
\begin{thm}\label{erroridentity}
For a function  
$n^*\in H(\Omega,\dive\Dive)$ satisfying the condition  (\ref{cond-N}) and a function $v\in \mathbb{K}$  the identity
\begin{equation}
\boldsymbol{\mu}(v)+\boldsymbol{\mu}^*(n^*)=\frac{1}{2}\|\nabla \nabla v-n^*\|^2+\int\limits_{\Omega}(f-\dive\Dive n^*)(\varphi -v)dx,
\label{7-2-14b}
\end{equation}
holds. The left-hand side of (\ref{7-2-14b}) is a measure of the deviation of  $v$ from $u$ and of $n^*$  from $p^*$, while the right-hand side of the above identity is a fully computable expression.
\end{thm}


 \subsection{Extension of the admissible set for $n^*$}
 
 Equality (\ref{7-2-14b})  provides a simple and transparent form of the error identity, but it operates with the
functions $n^* \in H(\Omega, \dive\Dive)$ satisfying the condition (\ref{cond-N}).   This functional set   is rather narrow and
 inconvenient if we wish to use in practise. In this subsection, we overcome this drawback and extend the admissible set for $n^*$.


\begin{lemma} \label{thin-L2-3}
For any function $\widetilde{n}^* \in H(\Omega, \dive\Dive )$ the projection inequality
\begin{equation} \label{difference_for_n}
\inf\limits_{n^*\in Q^*_{\ominus}}\|n^*-\widetilde{n}^*\| \leqslant 
C_{F_{\Omega}}\|( f-\dive\Dive \widetilde{n}^*)_{\oplus}\|
\end{equation}
holds. Here $C_{F_{\Omega}}$ is the constant defined by (\ref{F-constant}), and the set $Q^*_{\ominus}$ is determined in (\ref{Qzmin}).
\end{lemma}

\begin{proof}
For any function $m^*\in H(\Omega, \dive\Dive )$ the equality
\begin{multline*} 
\sup\limits_{w\in V^+(\Omega)}\!\int\limits_{\Omega}\!\left(\!\frac{1}{2}|m^*\!-\!\widetilde{n}^*|^2+fw-m^* \!: \nabla \nabla w \right)dx\\
 =\!\frac12\|m^*\!-\!\widetilde{n}^*\|^2+\!\!
\sup\limits_{w\in V^+(\Omega)}\!\int\limits_{\Omega}\!(f\!-\!\dive\Dive m^*)  w dx\\
=
\left\{
\begin{array}{cc}
\!\!\!\frac12\|m^*-\widetilde{n}^*\|^2,\,&\text{\rm if}\,m^*\in Q^*_\ominus,\\
+\infty, \,&\text{\rm if}\,m^*\not\in Q^*_\ominus.
\end{array}
\right.
\end{multline*}
holds. Therefore
\begin{equation} \label{inf_sup}
\inf\limits_{n^*\in Q^*_{\ominus}}\frac{1}{2}\|n^*-\widetilde{n}^*\|^2 =
\inf\limits_{n^* \in N^*} \sup\limits_{w\in H^{2,+}_0(\Omega)}\int\limits_{\Omega}\left(\frac{1}{2}|n^*-\widetilde{n}^*|^2+fw-n^* : \nabla \nabla w \right)dx.
\end{equation}
The Lagrangian defining the minimax formulation (\ref{inf_sup}) is linear w.r.t. $w$ and convex w.r.t. $n^*$. For $w=0$ it is coercive w.r.t. the dual variable. The space $N^*$ is a Hilbert one, and $V^+(\Omega)$ is a convex closed subset of the reflexive space $V$. Using the well-known
sufficient conditions providing the possibility of permutation $\inf$ and sup (see, for example, \cite {ET76}, \S 2 of chapter IV), we can rewrite (\ref {inf_sup}) in the form
\begin{equation*}
\inf\limits_{n^*\in Q^*_{\ominus}}\frac{1}{2}\|n^*-\widetilde{n}^*\|^2 
=\sup\limits_{w\in V^+(\Omega)} \inf\limits_{n^* \in N^*} 
\int\limits_{\Omega}\left(\frac{1}{2}|n^*-\widetilde{n}^*|^2+fw-n^* : \nabla \nabla w \right)dx.
\end{equation*}

Examination of the infimum w.r.t. $n^* \in N^*$ is reduced to an algebraic problem whose solution satisfies the equation $n^*=\widetilde{n}^*+\nabla \nabla w$ almost everywhere in $\Omega$. Using this equation and integrating by parts we obtain
\begin{equation} \label{kappa}
\begin{aligned}
\inf\limits_{n^*\in Q^*_{\ominus}}\frac{1}{2}\|n^*-\widetilde{n}^*\|^2 
&=\sup\limits_{w\in V^+(\Omega)}
\int\limits_{\Omega}
\left(-\frac{1}{2}|\nabla\nabla w|^2+fw-\widetilde{n}^* : \nabla \nabla w \right)dx\\
&=\sup\limits_{w\in V^+(\Omega)}
\int\limits_{\Omega} 
\left(-\frac{1}{2}|\nabla\nabla w|^2 +(f-\dive\Dive \widetilde{n}^*)w\right)dx\\
&\leqslant \sup\limits_{w\in V^+(\Omega)}
\int\limits_{\Omega} 
\left(-\frac{1}{2}|\nabla\nabla w|^2 +(f-\dive\Dive \widetilde{n}^*)_{\oplus}w\right)dx.
\end{aligned}
\end{equation}
Successive application of the Friedrich's type inequality
\begin{equation} \label{F-constant}
\|w\|\leqslant C_{F_{\Omega}}\|\nabla w\|\leqslant C^2_{F_{\Omega}}\|\nabla\nabla w\|
\end{equation}
allows us to estimate the last integral as follows
\begin{equation} \label{estimate_w}
\int\limits_{\Omega} (\dive\Dive \widetilde{n}^*+f)_{\oplus}wdx \leqslant C^2_{F_{\Omega}} 
\|(f-\dive\Dive \widetilde{n}^*)_{\oplus}\|\|\nabla \nabla w\|.
\end{equation}

Denoting $t:=\|\nabla \nabla w\|$ and combining (\ref{kappa}) with (\ref{estimate_w}), we see that supremum in (\ref{kappa}) can be estimated from above by the quantity
$$
 \sup\limits_{t \geqslant 0} \left(
-\frac{1}{2}t^2+C^2_{F_{\Omega}}  \|(f-\dive\Dive \widetilde{n}^*)_{\oplus}\| t\right)=\frac{1}{2}C^2_{F_{\Omega}}  \|(f-\dive\Dive \widetilde{n}^*)_{\oplus}\|^2 
$$
and complete the proof.
\end{proof}

\subsection{Error majorant}

Now we use Lemma~\ref{thin-L2-3} and identity (\ref{7-2-14b}) to obtain an estimate that holds for $\wt n^*\in H(\Omega, \dive\Dive )$. First of all, we have to transform the expressions for measure $\boldsymbol{\mu}^*(n^*)$ defined by the formula (\ref{error3}). Using the Young inequality (with the parameter $\beta$) we get 
the following lower bound for 
 $\boldsymbol{\mu}^*(n^*)$:
\begin{equation}\label{fm-1}
\begin{aligned}
 \boldsymbol{\mu}^*(n^*)&=
 \frac{1}{2}\|p^*-n^*+\widetilde{n}^*-\widetilde{n}^*\|^2+\int_{\Omega \setminus \{u=\varphi \}}(f-\dive\Dive (n^*-\widetilde{n}^*+\widetilde{n}^*))(\varphi -u)dx \\
 &\geqslant \boldsymbol{\mu}^*_{\beta}(\widetilde{n}^*)
 +\frac{1}{2}
 \left(1-\frac{1}{\beta}\right)\|\widetilde{n}^*-n^*\|^2-
 \int\limits_{\Omega \setminus \{u=\varphi\}} (\varphi-u)\dive\Dive(n^*-\widetilde{n}^*)dx,
 \end{aligned}
\end{equation}
where 
$$\boldsymbol{\mu}^*_{\beta}(\widetilde{n}^*)=\dfrac{(1-\beta)}{2}\,\|p^*-\widetilde{n}^*\|^2+\mu^*_{\varphi}(\widetilde{n}^*).
$$
We turn to (\ref{7-2-14b}). For the right-hand side
we obtain the upper bound
\begin{multline}
\label{fm-2}
\boldsymbol{\mu}(v)+\boldsymbol{\mu}^*(n^*) =\frac{1}{2}\|\nabla \nabla v -n^*\|^2\\
+\int\limits_{\Omega}(f-\dive\Dive \widetilde{n}^*)(\varphi-v)dx
+\int\limits_{\Omega}(f-\dive\Dive (n^*-\widetilde{n}^*))(\varphi-v)dx
\\
 \leqslant {\frac12}(1+\beta)\|\nabla\nabla v-\widetilde{n}^*\|^2+{\frac12}\left(1+\frac{1}{\beta}\right)\|n^*-\widetilde{n}^*\|^2\\
+\int\limits_{\Omega}(f-\dive\Dive \widetilde{n}^*)(\varphi -v)dx-\int\limits_{\Omega}(\varphi -v)
\dive\Dive (n^*-\widetilde{n}^*)dx.
\end{multline}
The bounds (\ref{fm-1}) and (\ref{fm-2}) are valid for any  $\widetilde{n}^*\in H(\Omega, \dive\Dive)$.

Putting together (\ref{fm-1}) and (\ref{fm-2}), shifting  the terms
$$
\frac{1}{2}\left(1-\frac{1}{\beta}\right)\|\widetilde{n}^*-n^*\|^2 \quad \text{and} \quad \int\limits_{\Omega \setminus \{u=\varphi\}}(\varphi -u)\dive\Dive (n^*-\widetilde{n}^*)dx
$$ to the opposite side, and combining the similar terms,  we get
\begin{equation}\label{fm-3}
\begin{aligned}
\boldsymbol{\mu}(v)+ \boldsymbol{\mu}_{\beta}^*(\widetilde{n}^*)& \leqslant \frac{1}{2}(1+\beta)\|\nabla \nabla v-\widetilde{n}^*\|^2+\int\limits_{\Omega} (f-\dive\Dive \widetilde{n}^*)(\varphi -v)dx\\
&+\frac{1}{\beta}\|n^*-\widetilde{n}^*\|^2
+\int\limits_{\Omega}(v-u)\dive\Dive(n^*-\widetilde{n}^*)dx.
\end{aligned}
\end{equation}

Successive application of H{\"o}lder's inequality and Young's inequality (with parameter $\beta$) to the last term on the right hand side of (\ref{fm-3}) yields the inequality
\begin{equation}\label{fm-4}
\int\limits_{\Omega}(v-u)\dive\Dive(n^*-\widetilde{n}^*)dx \leqslant
\frac{\beta}{2}\|\nabla \nabla (u-v)\|^2+\frac{1}{2\beta}\|n^*-\widetilde{n}^*\|^2.
\end{equation}

Relations (\ref{fm-3}), (\ref{fm-4}), and (\ref{difference_for_n}) provide the desired estimate (\ref{2-22}) where the right-hand side contains only known functions and can be computed explicitly.


\begin{thm}\label{first_majorant}
For any $v\in \mathbb{K}$ and  $\widetilde{n}^*\in H(\Omega, \dive\Dive )$, the full error measure is subject to the estimate
\begin{equation} \label{2-22}
\frac{1-\beta}{2}\bigg(\|\nabla\nabla (u-v)\|^2 
+\|p^*-\widetilde{n}^*\|^2\bigg) +\mu_{\varphi}(v)+\mu_{\varphi}^*(\widetilde{n}^*)
\leqslant \mathfrak{M}(v, \widetilde{n}^*,f, \varphi, \beta ),
\end{equation}
where
\begin{align*}
\mathfrak{M}(v, \widetilde{n}^*,f, \varphi, \beta )&:=
\frac{1}{2}(1+\beta)\|\nabla\nabla v-\widetilde{n}^*\|^2+\frac{3}{2\beta}
C^2_{F_{\Omega}}\|(f-\dive\Dive \widetilde{n}^*)_{\oplus}\|^2\\
&+\int\limits_{\Omega}(f-\dive\Dive \widetilde{n}^*)(\varphi-v)dx,
\end{align*}
a  parameter $\beta \in (0,1)$, and $C_{F_{\Omega}}$ is the same constant as in Lemma~\ref{thin-L2-3}.
\end{thm}

\begin{remark}
In (\ref{fm-1}) - (\ref{fm-4}), we used Young's inequality with the same constant $\beta$. In general, the constants can be taken different. Then, after an optimization
(with respect to the constants) we get a more accurate (but also more cumbersome) expression for the majorant $\mathfrak{M}$, which we do not list here.
\end{remark}

\section{Numerical examples}

In this section, we consider two examples that demonstrate how the identity (\ref{7-2-14b}) and the estimate (\ref{2-22})  work in practice.

First, we consider a model 1D problem, where the exact solution is known and, therefore, we can explicitly compute approximation errors associated with the primal and dual variables. 
In this example, the
approximate solution has essentially smaller coincidence set than the exact one. Nevertheless the error identity holds and error estimates computed for a regularized dual approximation are quite sharp.

Another example is motivated by an obstacle problem with a plane obstacle for radially symmetric plate which is fixed on the boundary. The obtained results are similar to those received in the 1D model problem and illustrate the validity of the error identity (\ref{7-2-14b}).

Certainly it will be interesting to apply these estimates for those cases, where approximations are constructed by some standard (e.g. FEM) approximations of the biharmonic obstacle problem. However, this question is beyond the framework of the present paper. We plan to devote a special paper to a detailed consideration of this question.

\subsection{Model 1D problem}

Let $\Omega=(-1,1)$, let  $\varphi \equiv -1$, and let  $f\equiv c$. For $c=-1152$\,  the minimizer of the problem (\ref{int}) has the form
$$
u(x)=\left\{
\begin{aligned}
&-8(x+1)^2(6x^2+4x+1), &&\text{\rm if}\quad  -1\leqslant x < -0.5,\\
&-1, &&\text{\rm if}\quad x\in \Omega_{\varphi}:=[-0.5, 0.5],\\
&-8(x-1)^2(6x^2-4x+1), && \text{\rm if}\qquad\,  0.5< x\leqslant 1.
\end{aligned}
\right.
$$
This function satisfies the boundary conditions $u(\pm 1)=u'(\pm 1)=0$ and the equation $u^{IV}+c=0$ in $\Omega_0=(-1,-0.5) \cup (0.5,1)$.
Notice also that 
$$
p^*= u''(x)=\left\{
\begin{aligned}
&-48(2x+1)(6x+5), &&\text{\rm if}\quad -1< x<-0.5,\\
&0, &&\text{\rm if}\qquad  x \in \Omega_{\varphi},\\
&-48(2x-1)(6x-5), &&\text{\rm if}\qquad 0.5<x < 1.
\end{aligned}
\right.
$$
The flux $p^*$ does not satisfy  (\ref{regularitynz}) which in this case reduces to $(n^*)^{\prime\prime} \in L^2(-1,1)$. Function $p^*$ and its derivative are shown in Fig.~\ref{Div_p_bild}. It is easy to see that   $\dive\Dive p^* \notin L^2(-1,1)$.

 \vskip-0.2cm
\begin{figure}[htbp]
\centering
\includegraphics[width=0.44\textwidth]{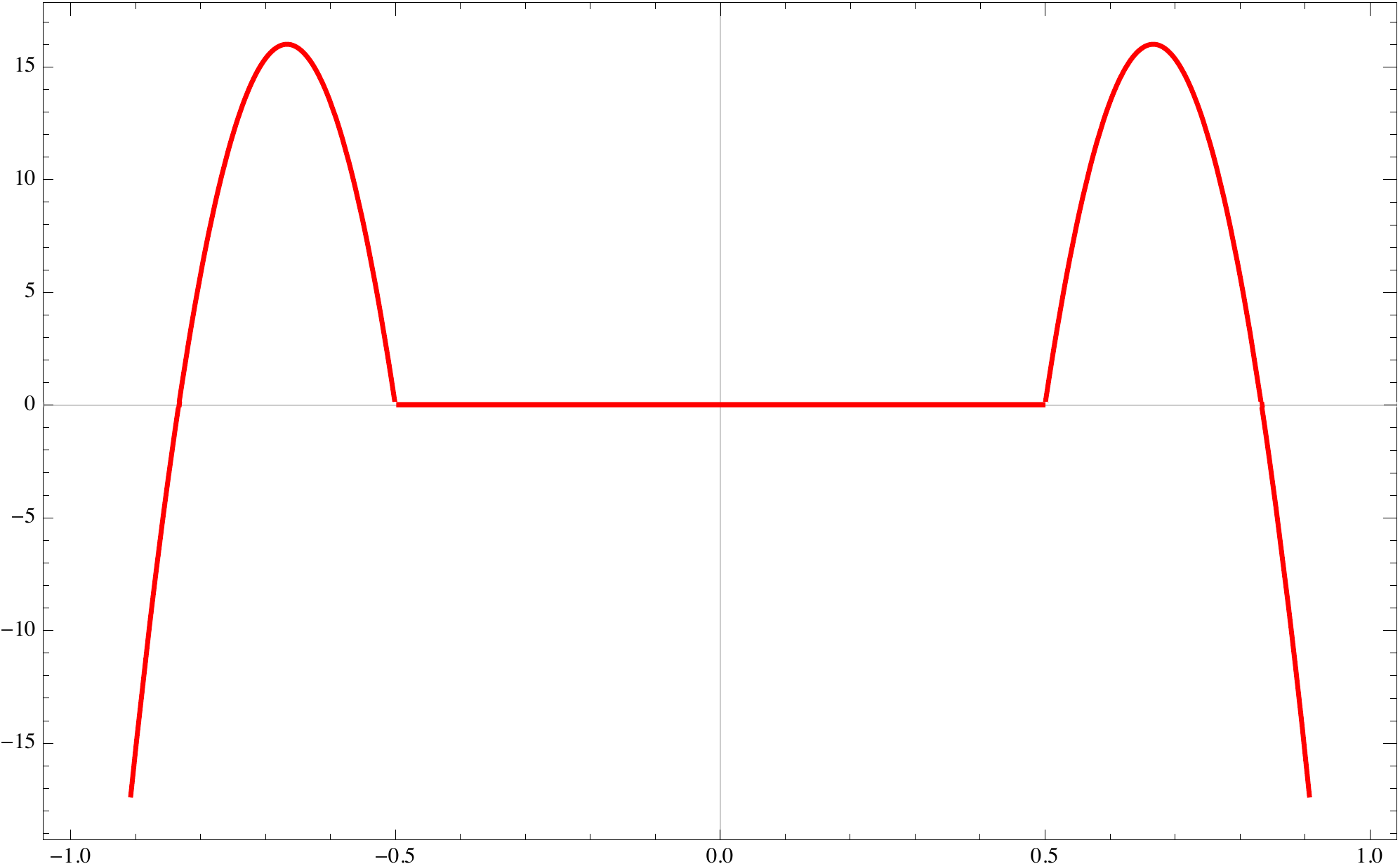}\qquad
\includegraphics[width=0.45\textwidth]{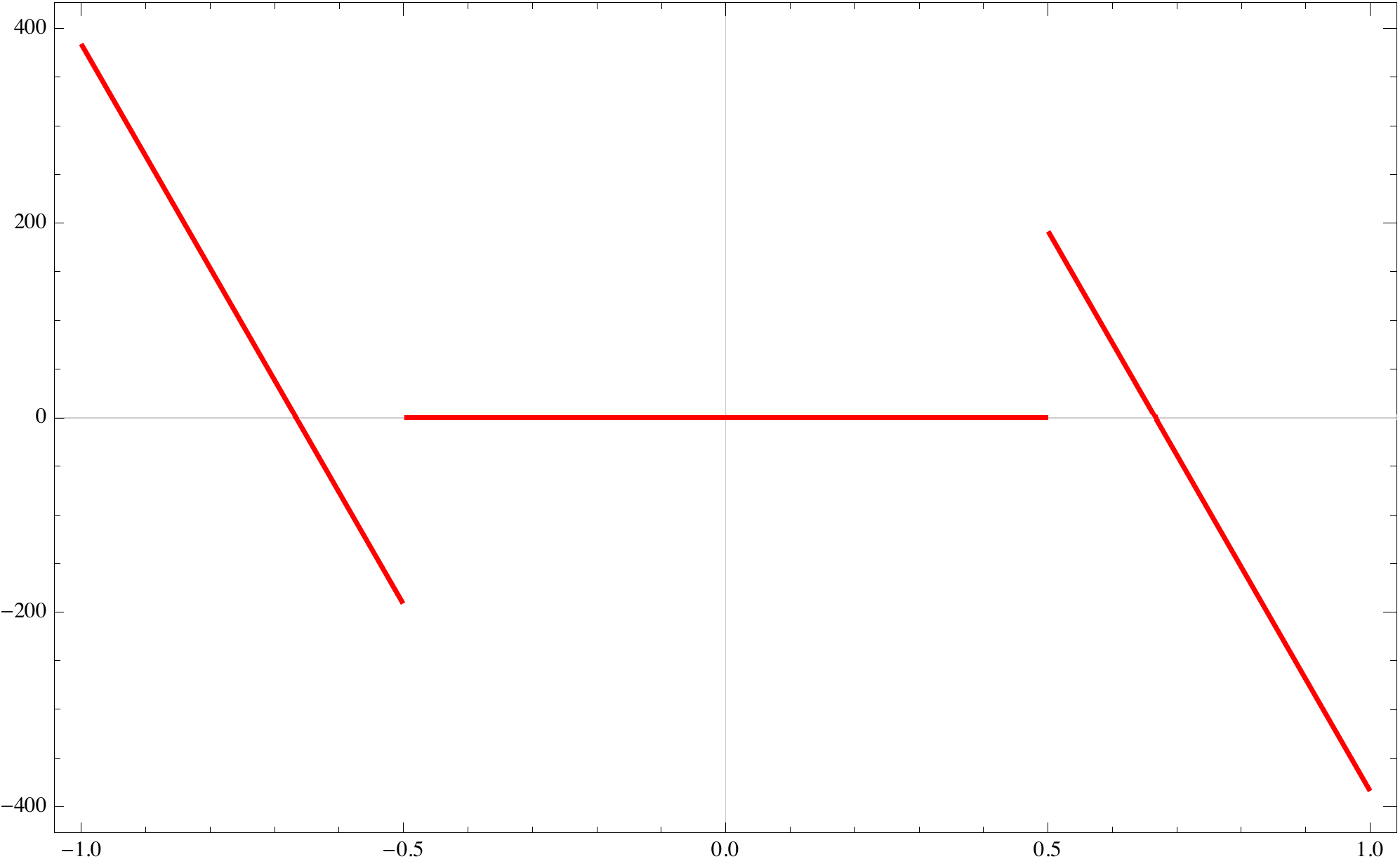}
\caption{The exact flux $p^*$ (left) and 
$(p^*)'$ (right)}
\label{Div_p_bild}
\end{figure}

\vspace{0.5cm}
Consider the function 
$$
v_1(x)=\left\{
\begin{aligned}
&-\frac{16}{27}(x+1)^2(1-8x), &&\text{\rm if}\quad -1 < x<-0.25,\\
&-1, &&\text{\rm if}\quad -0.25 \leqslant x\leqslant 0.25,\\
&-\frac{16}{27}(x-1)^2(1+8x), &&\text{\rm if}\qquad\, 0.25<x < 1.
\end{aligned}
\right.
$$
Obviously, $v_1 \in \mathbb{K}$ and $\{x\in \Omega \mid v_1(x)=-1\} \subset  \Omega_{\varphi}$ (see Fig.~\ref{Exact-Solution-bild22}).

\vspace{0.2cm}
As an approximation of the flux, we first consider the function
$$
n^*(x)=\left\{
\begin{aligned}
&20(2x+1)^2(5+6x), &&\text{\rm if}\quad -1 < x<-0.5,\\
&0, &&\text{\rm if}\quad -0.5 \leqslant x\leqslant 0.5,\\
&20(2x-1)^2(5-6x), &&\text{\rm if}\qquad\, 0.5<x < 1.
\end{aligned}
\right.
$$
 \vskip-0.2cm
\begin{figure}[htbp]
\centering
\includegraphics[scale=0.8]{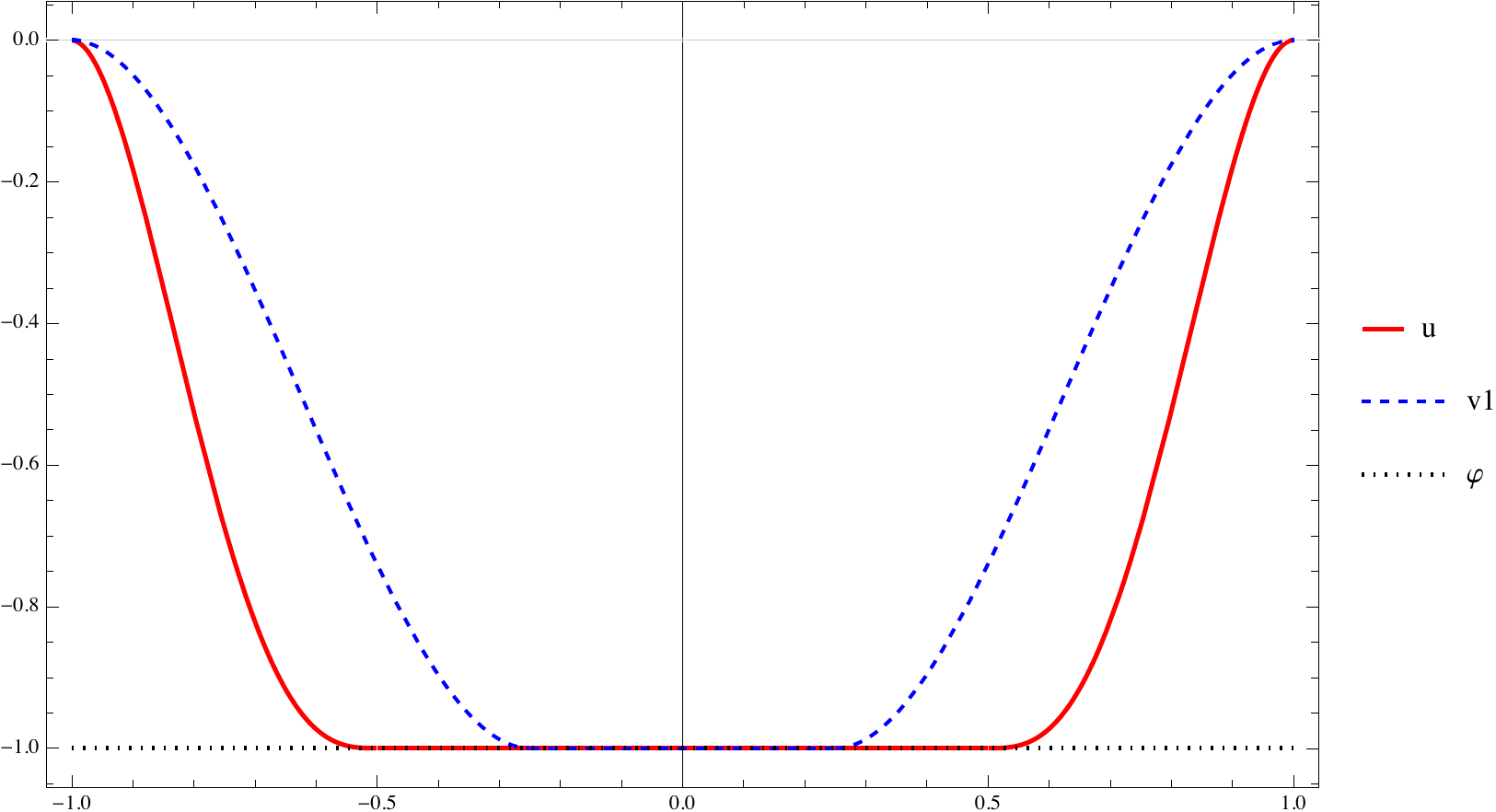}
\caption{The exact  1D-solution $u$  and the approximation solution
$v_1$ }
\label{Exact-Solution-bild22}
\end{figure}
Notice that $n^*$ satisfies the conditions (\ref{regularitynz})  and (\ref{cond-N}) (see Fig.~\ref{Div_n_bild}). According to (\ref{error2}) and (\ref{measure1}) the measure $\boldsymbol{\mu}(v_1)$ consists of two terms. In this example they can be calculated:
\begin{equation} \label{1_term_mu_v2}
\frac{1}{2}\|\nabla\nabla (u-v_1)\|^2=\int\limits_0^1 (u''-v_1'')^2dx=\int\limits_{0.25}^{0.5}(v_1'')^2dx+\int\limits_{0.5}^1(u''-v_1'')^2dx\simeq 125.15
\end{equation}
and
\begin{multline} 
\label{2_term_mu_v2}
\mu_{\varphi}(v_1)=\int\limits_{-0.5}^{0.5}(\dive\Dive\nabla\nabla\varphi  -f)(v_1-\varphi)dx
\\ -\bigl[u'''(-0.5)\bigr](v_1-u)\big|_{x=-0.5}-\bigl[u'''(0.5)\bigr](v_1-u)\big|_{x=0.5}
=152.89.           
\end{multline}
Here, $\bigl[u'''(a)\bigr]:=\big\{u'''(a-0)-u'''(a+0)\big\}$ denotes the jump at the point $a$.
\vskip-0.1cm
\begin{figure}[htbp]
\centering
\includegraphics[width=0.43\textwidth]{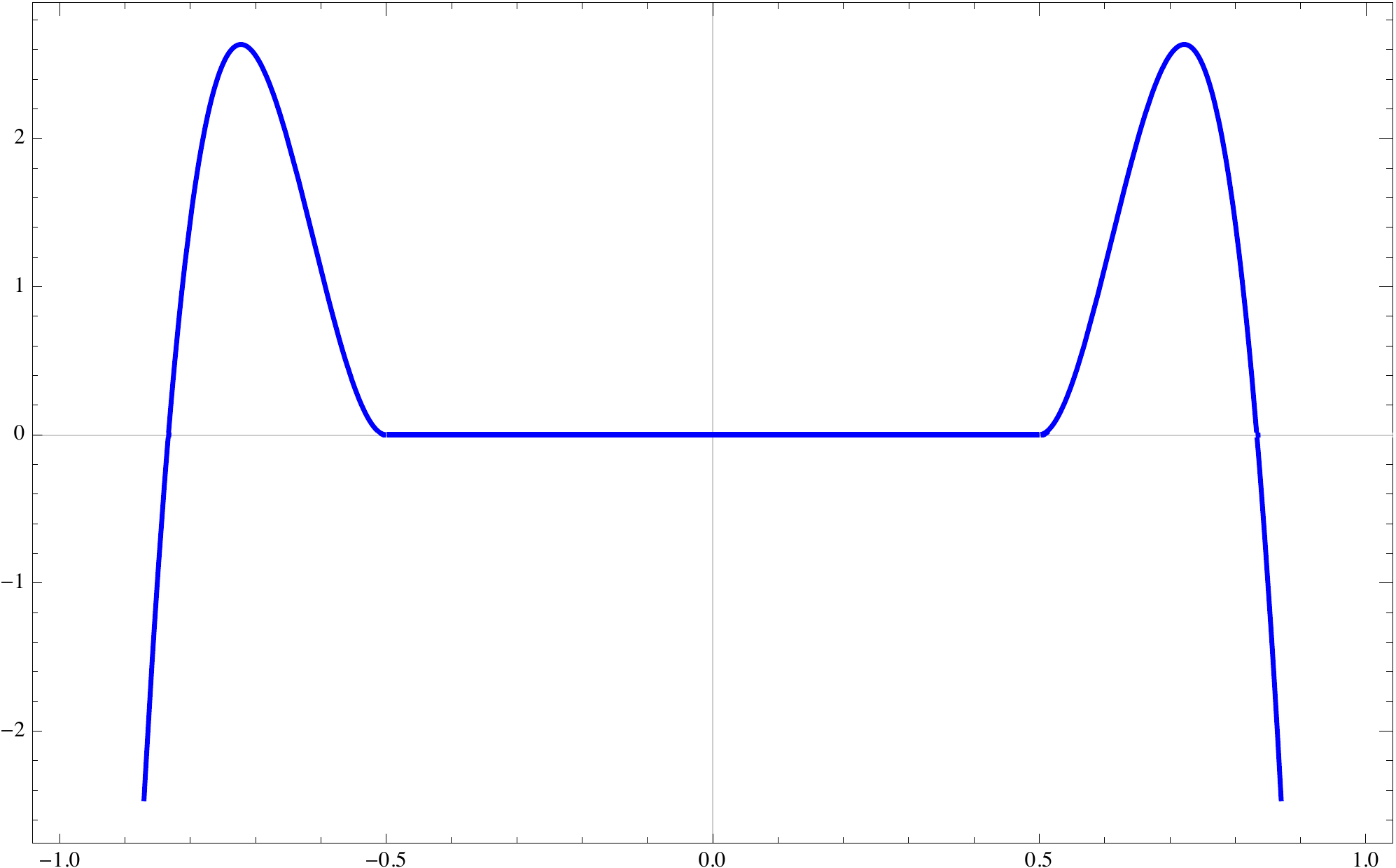}\qquad
\includegraphics[width=0.43\textwidth]{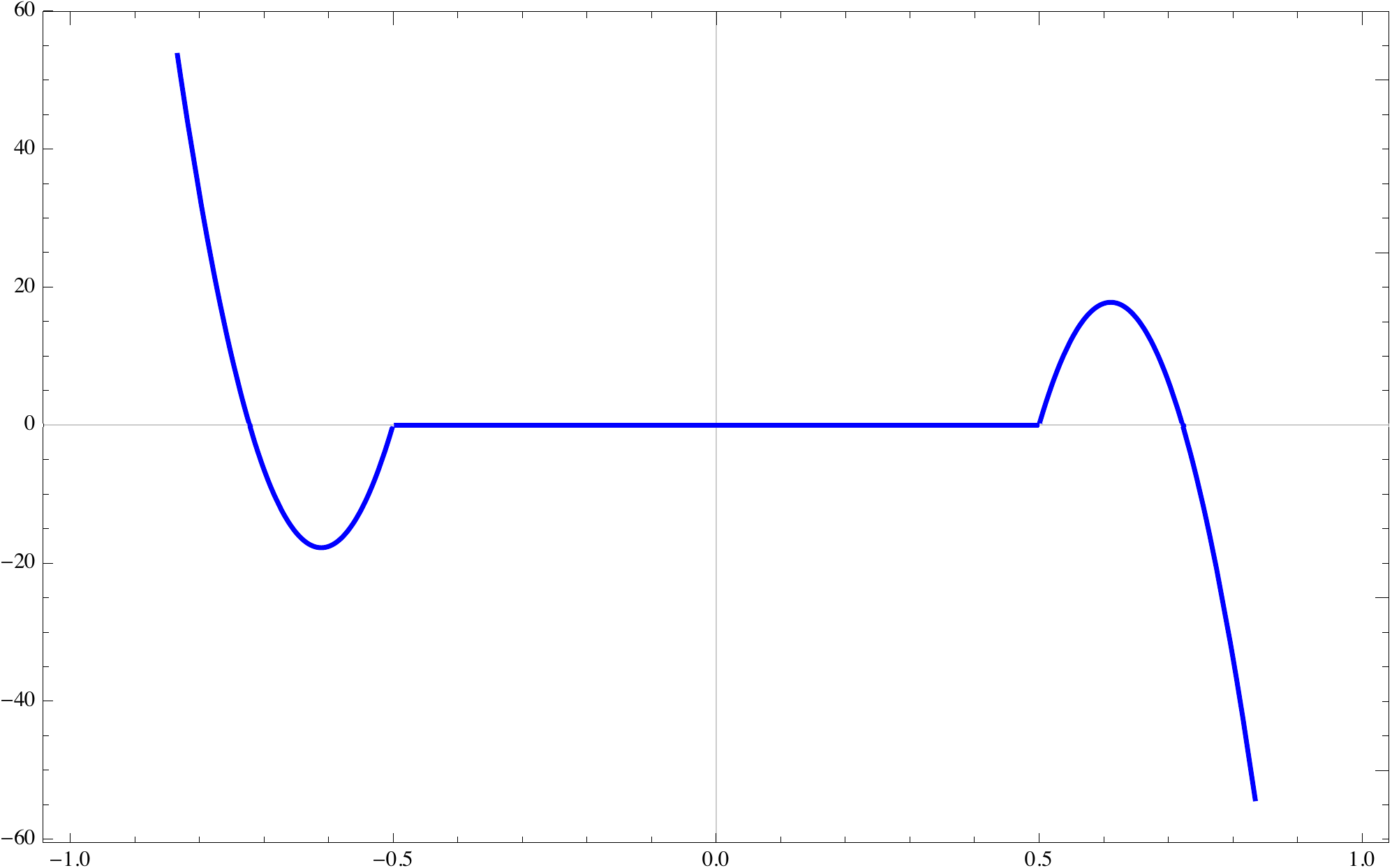}
\caption{The function $n^*$ (left) and 
$(n^*)'$ (right)}
\label{Div_n_bild}
\end{figure}



The measure ${\boldsymbol \mu^*}(n^*)$ is calculated in accordance with  (\ref{error3}) and (\ref{measure2}):
\begin{eqnarray}
 \label{1_term_mu*}
&&\frac{1}{2}\|p^*-n^*\|^2=\int\limits_{0.5}^1(p^*-n^*)^2dx\simeq 74.74,\\
 \label{2_term_mu*}
&&\mu^*_{\varphi}(n^*)=
2\int\limits_{0.5}^1(f\!-\!n'')(\varphi -u)dx\simeq 156.8.
\end{eqnarray}

Combination of (\ref{1_term_mu_v2}), (\ref{2_term_mu_v2}), (\ref{1_term_mu*}), and (\ref{2_term_mu*}) implies the following full error measure for deviations of the functions $v_1$ and $n^*$ from the exact solutions of direct and dual problems, respectively. We obtain
\begin{equation} \label{lhs_for_v1}
\boldsymbol{\mu}(v_1)+\boldsymbol{\mu}^*(n^*)
=278.04+231.54=509.58.
\end{equation}
Consider the right-hand side of the error identity (\ref{7-2-14b}) (since the chosen function $n^*$ satisfies the condition  (\ref{cond-N}) this identity holds).
Direct calculation 
yields
\begin{equation*}
\frac{1}{2}\|\nabla\nabla v_1-n^*\|^2=\int\limits_0^1(v_1''-n^*)^2dx=\int\limits_{0.25}^{0.5}(v_1'')^2dx+\int\limits_{0.5}^1(v_1''-n^*)^2dx
\simeq 23.063
\end{equation*}
and
\begin{equation*}
\int\limits_{-1}^1\!(f-\dive\Dive n^*)(\varphi-v_1)dx\!=\!2\!\int\limits_{0.25}^{0.5}\!f(\varphi\! -v_1)dx+2\!\int\limits_{0.5}^1\!(f-n'')(\varphi \!-\!v_1)dx
\simeq 486.515.
\end{equation*}
Thus, the sum of these terms gives the same value 509.58 as the sum of measures (\ref{lhs_for_v1}).

\smallskip

Next, we take $\tilde{n}^*$ such that $\dive\Dive \tilde{n}^* \in L^2(-1,1)$ but $\tilde{n}^*$ does not satisfy the condition (\ref{cond-N}). Set $\tilde{n}^*$ by the formula
$$
\tilde{n}^*(x)=\left\{
\begin{aligned}
&8(3x+1)^2(6x+5), &&\text{\rm if}\quad -1\leqslant x < -1/3,\\
&0, &&\text{\rm if}\quad -1/3 \leqslant x\leqslant 1/3,\\
&8(3x-1)^2(5-6x), &&\text{\rm if}\qquad\, 1/3<x\leqslant 1.
\end{aligned}
\right.
$$
On the set $x\in (-1,-\frac{17}{18}] \cup [\frac{17}{18},1)$ the condition  (\ref{cond-N}) does not hold.
Therefore, we  cannot use the error identity (\ref{7-2-14b}) but can use the estimate (\ref{2-22}). Let us verify how accurately it holds.

By direct calculations we obtain
\begin{align*}
&\frac{1}{2}\|p^*-\tilde{n}^*\|^2=\int\limits_{1/3}^{0.5}(\tilde{n}^*)^2dx+\int\limits_{0.5}^1(p^*-\tilde{n}^*)^2dx\simeq 24.9137,\\
&\mu_{\varphi}^*(\tilde{n}^*)=2\int\limits_{0.5}^1(f-\dive\Dive \tilde{n}^*)(\varphi-u)dx=2\int\limits_{0.5}^1(f-(\tilde{n}^{*})'')(\varphi-u)dx\simeq 72.0,\\
&\|\nabla\nabla v_1-\tilde{n}^*\|^2=2\int\limits_{0.25}^{1/3}(v_1'')^2dx+2\int\limits_{1/3}^1(v_1''-\tilde{n}^*)^2dx\simeq 66.16,\\
&\frac{1}{2}\|(f-\dive\Dive \widetilde{n}^*)_{\oplus}\|^2=
\int\limits_{17/18}^1(f-(\tilde{n}^{*})'')^2dx\simeq 384.0,\\
&\int\limits_{-1}^1(f-\dive\Dive \tilde{n}^*)(\varphi-v_1)dx=
2\int\limits_{1/3}^{0.5}f(\varphi-v_1)dx+2\int\limits_{0.5}^1(f-(\tilde{n}^{*})'')(\varphi-v_1)dx\\
&\simeq 51.0947+268.267\simeq 319.36.
\end{align*}

Recall that for $\Omega=(-1,1)$ we have 
$C_{F_{\Omega}}=4/\pi^2$. Thus, according to (\ref{2-22}) for any $\beta \in (0,1]$ the majorant $\mathfrak{M}(v_1,\tilde{n}^*,f,\varphi, \beta)$ has the form
\begin{equation*} 
\begin{aligned}
\mathfrak{M}(v_1,\tilde{n}^*,-1152,-1,{\beta})
\simeq {352.44 +33.08\cdot \beta+189.22\cdot \frac{1}{\beta}}.
\end{aligned}
\end{equation*}

Taking into account 
(\ref{1_term_mu_v2}) and (\ref{2_term_mu_v2}), we get  the  expression 
$524.95-150.06\cdot \beta$
 for the left-hand side of the inequality (\ref{2-22}).
Thus,  for any $\beta \in (0,1]$, this inequality takes the form 
\begin{equation}
\label{examp1}
524.95-150.06\, \cdot\beta \leqslant 352.44 +33.08\,\cdot\beta+189.22\, \cdot\frac{1}{\beta}.
\end{equation} 
 
In particular, for $\beta=0.5$ and $\beta=1$ the ratio of the majorant (the r.h.s. of  (\ref{examp1})) to the deviation measure (the l.h.s. of (\ref{examp1}))
is characterized by the values 1.66 and 1.53, respectively.
\vspace{0.05cm}

\vskip-0.2cm
\begin{figure}[htbp!]
\centering
\includegraphics[scale=0.85]{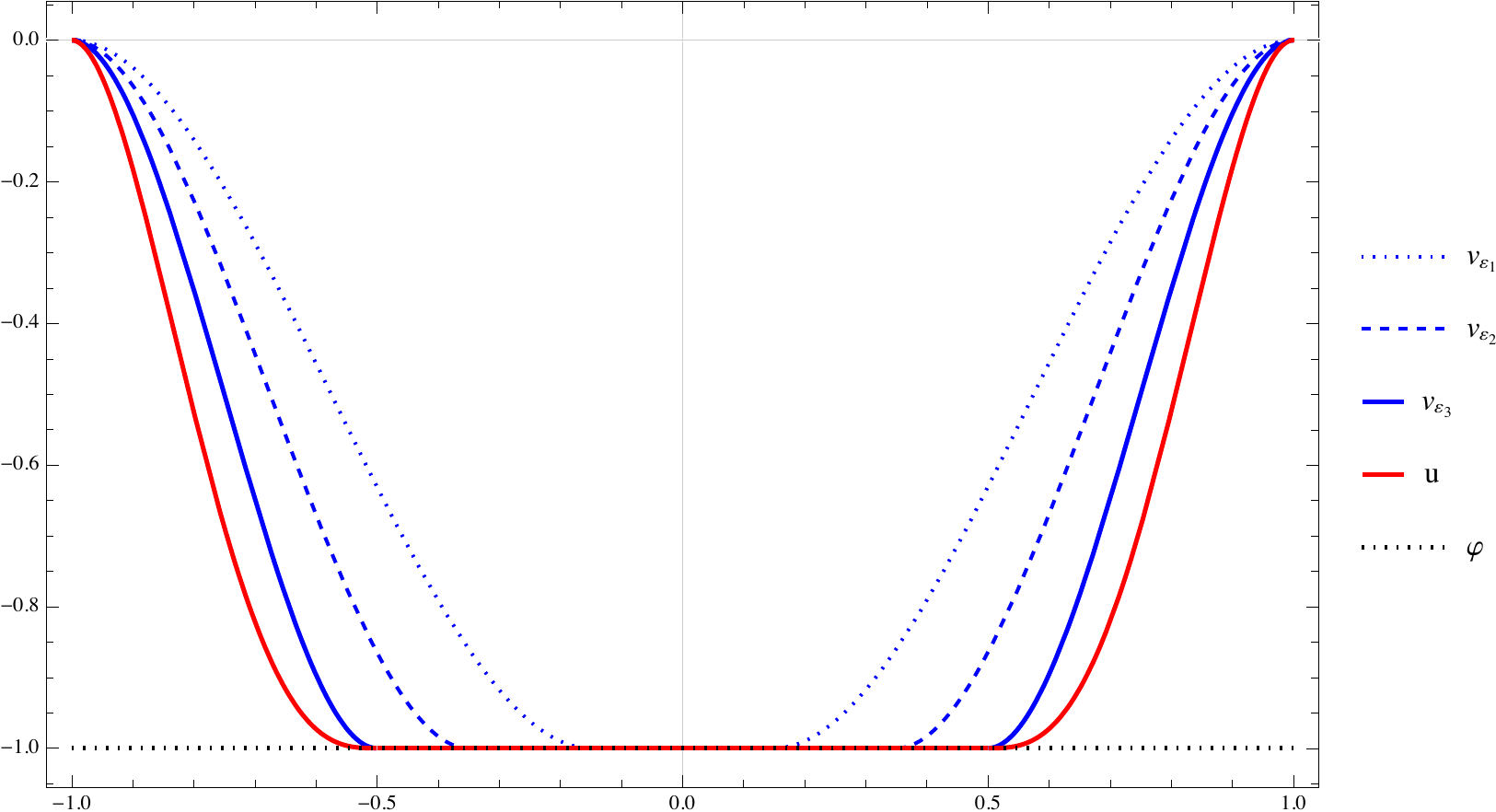}
\caption{The minimizer $u$  and the functions 
$v_{\varepsilon}$ for $\varepsilon_1= 0.35$, $\varepsilon_2=0.15$, $k_2=3$, and $\varepsilon_3=0$. }
\label{Exact-Solution-bild2}
\end{figure}

Further, we consider the approximations (see Fig.~~\ref{Exact-Solution-bild2})
$$
v_{\varepsilon}(x)=\left\{
\begin{aligned}
&-\frac{4}{(2\varepsilon+1)^3}(1+x)^2(-4x+6\varepsilon-1), &&\text{if}\quad -1 \leqslant x < \varepsilon-1/2,\\
&-1, &&\text{if}\quad \varepsilon-1/2\leqslant x \leqslant 1/2-\varepsilon,\\
&-\frac{4}{(2\varepsilon+1)^3}(1-x)^2(4x+6\varepsilon-1), &&\text{if}\quad 1/2-\varepsilon <x \leqslant 1,
\end{aligned}
\right.
$$
where $\varepsilon$ is a parameter satisfying for $0 \leqslant \varepsilon \leqslant 1/2$.
For these approximations we have
$\{x\in \Omega \mid v_{\varepsilon}=-1\} \subset  
\Omega_{\varphi}$.
Notice that
we get a better approach of the coincidence set $\{x\in \Omega \mid u=-1\}$ as $\varepsilon \to 0$. 
Moreover, the function $v_0$ coincides with $\varphi$ on the exact coincidence set, but the function $v_0$ does not coincides with $u$.
\vspace{0.2cm}

Approximations $n_{\varepsilon}^*$ of the exact flux $p^*$ 
are constructed by smoothing the second derivative of $v_{\varepsilon}$ (which replace $\nabla\nabla v_{\varepsilon}$) such that  $(n_{\varepsilon}^*)^{\prime\prime} \in L^2(-1,1)$. The latter corresponds to the condition $n_{\varepsilon}^*\in H(\Omega, \dive\Dive)$.
In particular, if we take 
$$
n_{\varepsilon}^*(x)=\left\{
\begin{aligned}
&\frac{48}{(2\varepsilon +1)^5}(-2x+2\varepsilon-1)(4x-2\varepsilon+3), &&\text{if}\quad -1\leqslant x <\varepsilon-1/2,\\ 
&0, &&\text{if}\quad \varepsilon-1/2 \leqslant x \leqslant 1/2-\varepsilon,\\ 
&\frac{48}{(2\varepsilon +1)^5}(2x+2\varepsilon-1)(-4x-2\varepsilon+3), &&\text{if}\quad 1/2-\varepsilon <x \leqslant 1,         
\end{aligned}
\right.
$$
then (\ref{cond-N}) is satisfied, and again we have to verify the validity of the error identity (\ref{7-2-14b}).

Table~\ref{tab:primal} contains results related to the components of $\boldsymbol{\mu}(v_{\varepsilon})$ computed for $\varepsilon=0.05j$, $j=7, 5, 3, 1$, and $0$. It shows that both  terms (quadratic and nonlinear) decrease as $\varepsilon \to 0$.
However, the first term remains positive, since the sequence of approximate solutions does not tend to the exact solution $u$, while the second term tends to zero, because the corresponding sequence of the approximated coincidence sets tends to the exact set $\Omega_{\varphi}$. It is easy to see that the sum of these two terms, constituting a measure $\boldsymbol{\mu}(v_{\varepsilon})$, is equal to the deviation of $J(v_\varepsilon)$  from the exact minimum of the direct variational problem.
In the last column of Table~\ref{tab:primal}, we present the relative contribution of the nonlinear measure $\mu_{\varphi}(v_{\varepsilon})$ expressed by the quantity
$$
k(v_{\epsilon}):=100 \frac{\mu_{\varphi}(v_{\varepsilon})}{\boldsymbol{\mu}(v_{\varepsilon})} \quad [\%].
$$

\begin{table}[htbp]
  \centering
  \begin{tabular}{ccccccccc}
  \hline
  \ & \ &\ &\ &\ &\ &\ &\ & \\
  $\varepsilon$ & \quad & $\frac{1}{2}\|\nabla\nabla (u-v_{\varepsilon})\|^2$& \quad & $\mu_{\varphi}(v_{\varepsilon})$&\quad & $J(v_{\varepsilon})-J(u)$& \quad &  
  $k(v_{\epsilon})\ [\%]$  \\
\ & \ &\ &\ &\ &\ &\ &\ & \\
  \hline
\ & \ &\ &\ &\ &\ &\ &\ & \\
  0.35 & \quad & 134.060 &\quad & 250.280 & \quad & 384.340 & \quad & 65.12\\ 
  0.25 & \quad & 125.156 & \quad & 152.889 & \quad & 278.044 & \quad & 54.99\\ 
  0.15 & \quad & 109.904 & \quad & 68.192 & \quad & 178.096
  & \quad & 38.29\\ 
  0.05 & \quad & 81.474  & \quad & 9.852 & \quad & 91.326 & \quad & 1.06\\ 
  0.00 & \quad & 57.60 & \quad & 0 & \quad & 57.60 & \quad & 0\\ 
 \ & \ &\ &\ &\ &\ &\ &\ & \\
  \hline
  \end{tabular}
  \caption{Components of the measure $\boldsymbol{\mu}(v_{\varepsilon})$.}
  \label{tab:primal}
\end{table}

Table~\ref{tab:dual} encompasses the components of $\boldsymbol{\mu}^*(n^*_{\varepsilon})$. 
As in the case of the measure $\mu_{\varphi}(v_{\varepsilon})$,
both quadratic and nonlinear terms decrease as $\varepsilon \to 0$. Nevertheless, this nonlinear measure does not tend to zero. The measure controls the violation of the equation $\dive\Dive n^*=f$  on the set $\Omega_0$. Function  $n^*_\varepsilon$ does not satisfy the latter equation for any $\varepsilon$, and, consequently, $\mu^*(n^*_\varepsilon)>0$. Moreover, in this case the measure $\mu^*_{\varphi}(n^*_{\varepsilon})$ evidently dominates the first term,
which confirmed by the relative contribution
$$
k(n^*_{\varepsilon}):=100 \frac{\mu^*_{\varphi}(n^*_{\varepsilon})}{\boldsymbol{\mu}^*(n^*_{\varepsilon})}\quad [\%]
$$
listed in the last column. 

\begin{table}[htbp]
  \centering
  \begin{tabular}{ccccccccc}
  \hline
  \ & \ &\ &\ &\ &\ &\ &\ & \\
  $\varepsilon$ & \quad & $\frac{1}{2}\|p^*-n_{\varepsilon}\|^2$& \quad & $\mu_{\varphi}^*(n^*_{\varepsilon})$&\quad & $\boldsymbol{\mu}^*(n^*_{\varepsilon})=I^*(p^*)-I^*(n^*_{\varepsilon})$& \quad &  
  $k(n^*_{\epsilon})\ [\%]$  \\
\ & \ &\ &\ &\ &\ &\ &\ & \\
  \hline
\ & \ &\ &\ &\ &\ &\ &\ & \\
  0.35 & \quad & 119.444  &\quad & 422.937 & \quad & 542.381 & \quad & 77.98 \\ 
  0.25 & \quad & 109.443 & \quad & 400.119 & \quad & 509.562 & \quad & 78.52\\ 
  0.15 & \quad & 95.972 & \quad & 357.378  & \quad & 453.349
  & \quad & 78.83\\ 
  0.05 & \quad &  78.510 & \quad & 270.053 & \quad & 348.563 & \quad & 77.48\\ 
  0.00 & \quad & 68.571 & \quad & 192  & \quad & 260.571 & \quad & 73.68\\ 
 \ & \ &\ &\ &\ &\ &\ &\ & \\
  \hline
  \end{tabular}
  \caption{Components of the measure $\boldsymbol{\mu^*}(n^*_{\varepsilon})$.}
  \label{tab:dual}
\end{table}

Table~\ref{tab:primal-dual} reports on the  error identity (\ref{7-2-14b}). First two columns correspond to the two terms forming the right-hand side of (\ref{7-2-14b}). Here, $\mathcal{N}(v_{\epsilon}, n^*_{\varepsilon})$ denotes the term $\int_{-1}^1 (f-(n^*_{\varepsilon})^{\prime\prime})(\varphi -v_{\varepsilon})dx$ corresponding to the summand $\int_{\Omega}(f-\dive\Dive n^*_{\varepsilon})(\varphi -v_{\varepsilon})dx$ in the identity. The sum of these terms is given in the third column. It coincides exactly with the sum of measures given in Tables 1 and 2. Observe that the values in the first two columns of Table 3 are computed directly by functions $v_{\varepsilon}$ and $n^*_{\varepsilon}$. These functions can be considered as approximate solutions constructed with the help of some computational procedure. Table~\ref{tab:primal-dual} shows that the sum
$$
\frac{1}{2}\|\nabla\nabla v_{\varepsilon}-n^*_{\varepsilon}\|^2+\mathcal{N}(v_{\varepsilon}, n^*_{\varepsilon})
$$
is a good and easily computed characteristic of the quality of approximate solutions.

\begin{table}
  \centering
  \begin{tabular}{ccccccccc}
  \hline
  \ & \ &\ &\ &\ &\ &\ &\ & \\
  $\varepsilon$ & \ & $\frac{1}{2}\|\nabla\nabla v_{\varepsilon}-n_{\varepsilon}\|^2$& \ & $\mathcal{N}(v_{\varepsilon}, n^*_{\varepsilon})$&  \ & r.h.s. of (\ref{7-2-14b})&\ & $\boldsymbol{\mu}(v_{\varepsilon})+\boldsymbol{\mu}^*(n^*_{\varepsilon})$  \\
\ & \ &\ &\ &\ &\ &\ & \\
  \hline
\ & \ &\ &\ &\ &\ &\ & \\
  0.35 & \ & 10.049  &\ & 916.672 & \ & 926.721 & \ & 926.721  \\ 
  0.25 & \ & 14.629 & \ & 772.978 & \ & 787.606 & \ &  
  787.606 \\ 
  0.15 & \ & 22.472 & \ & 608.973  & \ & 631.445 & \ & 631.445   \\ 
  0.05 & \ & 37.094  & \ & 402.796 & \ & 439.890  & \ & 439.890 \\ 
  0.00 & \ & 49.371 & \ & 268.800 & \ & 318.171 & \ & 318.171 \\ 
 \ & \ &\ &\ &\ &\ &\ & \\
  \hline
  \end{tabular}
  \caption{Components of the identity (\ref{7-2-14b}) computed for $v_{\varepsilon}$ and $n^*_{\varepsilon}$ for different $\varepsilon$.}
  \label{tab:primal-dual}
\end{table}

\subsection{Bending of a circular plate} 

Let $\Omega=B_3 \subset \mathbb{R}^2$, where $B_3$ denotes the open ball with the center at the origin and radius $3$. In this case, the problem $\mathcal P$ can be considered as a simplified version of the bending problem for a clamped elastic circular plate above the plane obstacle $\varphi \equiv -1$ under the action of an external force $f$. Simplification is that we replace the tensor of the elastic constants by the unit tensor. In the context of the considered issues such a simplification does not play a significant role. We set
$$
f\equiv 16c_1/c_2, \quad \text{where}\quad c_1=9\ln{3}-4, \qquad c_2=208-216\ln{3}+9\ln^2{3}.
$$

Notice that for the given data we can explicitly define the radial solution of the problem (\ref{int}). In the polar coordinates $(r, \theta)$, the minimizer has  the following form:
$$
u(r, \theta)=\left\{
\begin{aligned}
&-1, &&\text{for}\ 0<r<1,\\
&\frac{(r^2-1)(128+c_1(r^2-3))+4(c_1-32(1+r^2))\ln{r}}{4c_2}-1, &&\text{for}\ 1\leqslant r \leqslant 3.
\end{aligned}
\right.
$$

It is clear that $\Delta^2u=f_0$ in $B_3\setminus B_1$, $u \geqslant -1$ in $B_3$, and $u(3, \theta)=\dfrac{\partial u}{\partial r}(3, \theta)=0$ for all $\theta \in [0,2\pi)$. 
An elementary calculation shows that
$$
p^*=\left(
\begin{array}{cc}
p_{11}^* & p_{12}^*\\
p_{21}^* & p_{22}^*
\end{array}
\right), 
$$
where $
p_{11}^*=p_{12}^*=p_{21}^*=p_{22}^*=0$ in the ball $B_1$, while for $(r, \theta) \in B_3\setminus B_1$, the components of $p^*$ are defined by the formulas 
\begin{align*}
p_{11}^*&=
\frac{(r^2-1)\cos{(2\theta)}(c_1(r^2+1)-32)+2r^2(c_1(r^2-1)-32\ln{r})}
{c_2r^2}, 
\\
p_{12}^*=
p_{21}^*&=
\frac{(r^2-1)\sin{(2\theta)}(c_1(r^2+1)-32)}{c_2r^2}, \\
p_{22}^*&=
\frac{(1-r^2)\cos{(2\theta)}(c_1(r^2+1)-32)+2r^2(c_1(r^2-1)-32\ln{r})}
{c_2r^2}. 
\end{align*}
It is easy to check that $\dive\Dive p^* \notin L^2(B_3)$ and $\dive\Dive p^*=f_0$ on the set $B_3 \setminus B_1=\Omega\setminus \{u=\varphi\}$.

Further, we define the function $v=v_2$  as follows:
$$
v_2(r, \theta)=u(r, \theta)+\left\{
\begin{aligned}
&0.5 [1-cos{(\pi (3-r))}], && \text{\rm if}\quad 1\leqslant r\leqslant 3,\\
&0, && \text{\rm if}\quad  0<r<1.
\end{aligned}
\right. 
$$
 It is clear that $v_2 \in \mathbb{K}$ and $v_2 \geqslant u$ in $\Omega$ and $v_2=u$ in $B_1$. Thus, $v_2$ has the same coincidence  set $B_1$ as the minimizer $u$ (see Fig.~\ref{Exact-Solution-bild}).
 \vskip-0.2cm
\begin{figure}[htbp]
\centering
\includegraphics[width=0.45\textwidth]{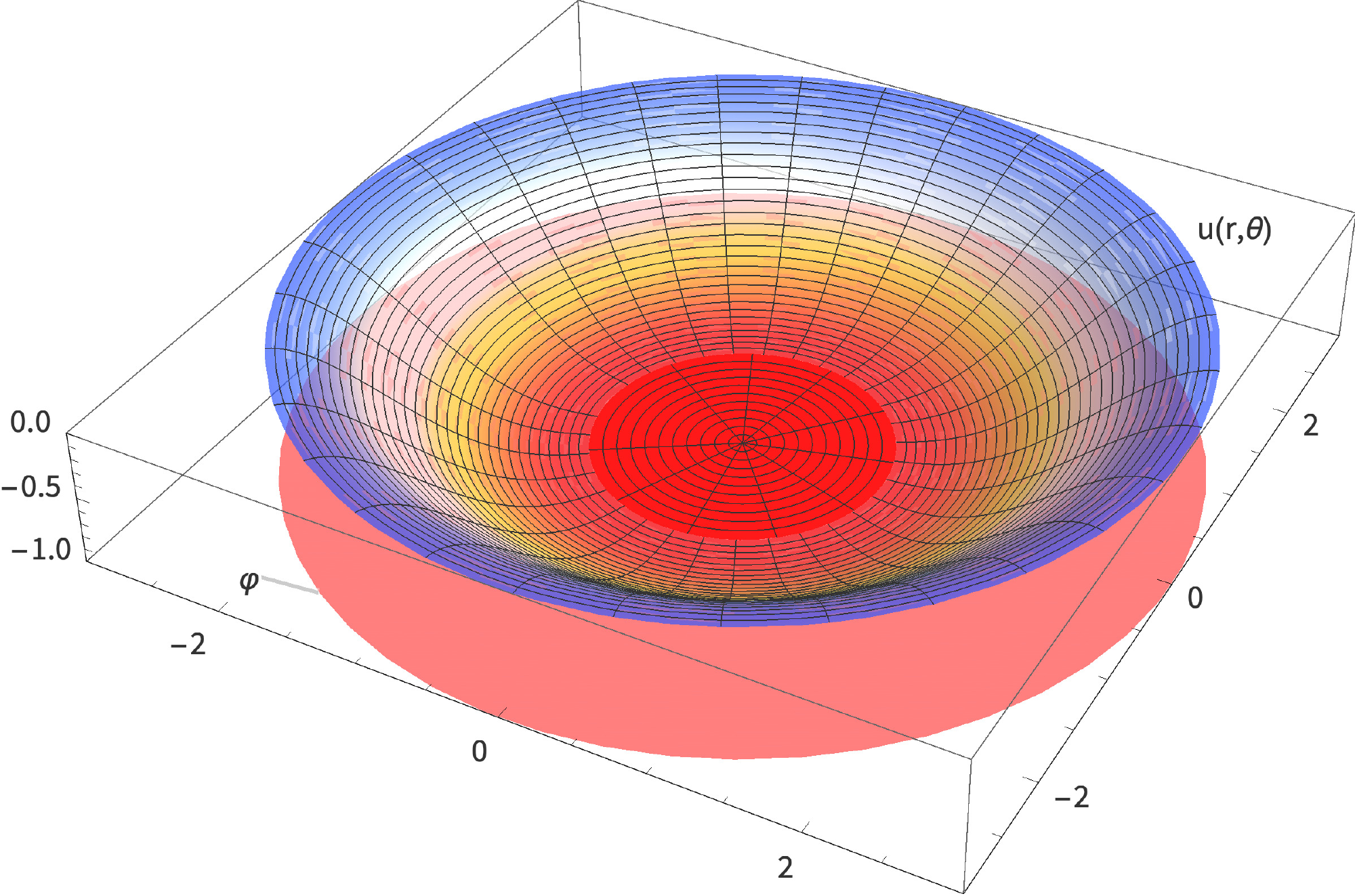}\qquad
\includegraphics[width=0.45\textwidth]{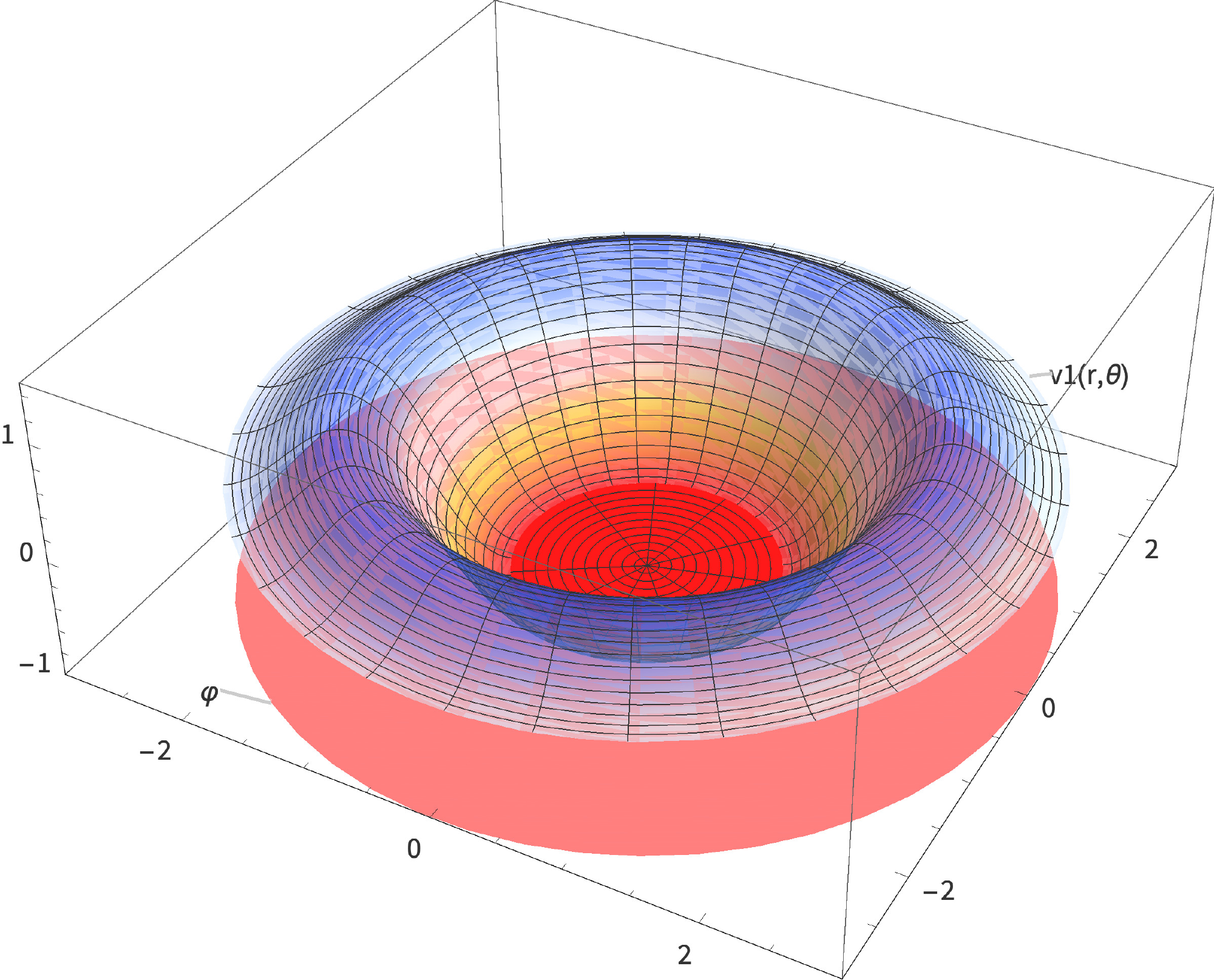}
\caption{The exact solution $u$ (left) and the function 
$v_2$ (right).}
\label{Exact-Solution-bild}
\end{figure}

By direct computations of the measure (\ref{error2}) we get 
\begin{equation} \label{mu_v1}
\boldsymbol{\mu}(v_2)=\frac{1}{2}\|\nabla \nabla (u-v_2)\|^2+\mu_{\varphi}(v_2)=\frac{1}{2}\|\nabla \nabla (u-v_2)\|^2 \simeq 157.19.
\end{equation}
We point out that $\mu_{\varphi}(v_2)=0$, since $\{x\in \Omega \mid v_2=-1\}=\Omega_\varphi$.

Let us set now
$$
\hat{n}^*=\left(
\begin{array}{cc}
\hat{n}_{11}^* & \hat{n}_{12}^*\\
\hat{n}_{21}^* & \hat{n}_{22}^*
\end{array}
\right),
$$
where 
\begin{align*}
\hat{n}_{11}^*=\hat{n}_{22}^*&=\left\{
\begin{aligned}
&0, && \text{for}\ (r, \theta) \in B_1,\\
&\frac{(r-1)^3\cos{(2\theta)}}{c_2r^2}, && \text{for}\ (r, \theta) \in B_3\setminus B_1, 
\end{aligned}
\right.\\
\hat{n}_{12}^*=
\hat{n}_{21}^*&=\left\{
\begin{aligned}
&0, && \text{for}\ (r, \theta) \in B_1,\\
&\frac{9(r-1)^3\sin{(2\theta)}}{c_2r^2}, && \text{for}\ (r, \theta) \in B_3\setminus B_1.
\end{aligned}
\right.
\end{align*}
The function $\hat{n}^*$
satisfies the condition (\ref{cond-N}). Therefore, 
we can verify the validity of the identity (\ref{7-2-14b}).

Computing the error measure $\boldsymbol{\mu}^*(\hat{n}^*)$ defined by (\ref{error3}) and (\ref{measure2}), we get
\begin{equation} \label{mu-n*}
\begin{aligned}
\boldsymbol{\mu}^*(\hat{n}^*)&=\frac{1}{2}\|p^*-\hat{n}^*\|^2+\int_{B_3 \setminus B_1} (f-\dive\Dive \hat{n}^*)(-1 -u)rdrd\theta\\
&\simeq 14.84+ 63.46 \simeq 78.30.
\end{aligned}
\end{equation}
Combination of (\ref{mu_v1}) and (\ref{mu-n*}) implies the following value for the full error measure (the result is rounded to two decimal places)
\begin{equation} \label{lhs-2-11}
\boldsymbol{\mu}(v_2)+\boldsymbol{\mu}^*(\hat{n}^*)\simeq 157.19+ 78.30 = 235.49.
\end{equation}

To compute the terms on the right-hand side of the identity 
(\ref{7-2-14b}), we use only the functions  $v_2$ and $\hat{n}^*$ (known approximations of the exact solutions). The sum of these terms gives the same value as (\ref{lhs-2-11}):
\begin{equation*} 
\frac{1}{2}\|\nabla \nabla v_2 -\hat{n}^*\|^2-
\int\limits_{B_3}(f-\dive\Dive \hat{n}^*)(1+v_2)rdr d\theta = 111.15 + 124.34 =235.49.
\end{equation*}

\subsection*{Acknowledgments}
D.A. was supported by the German Research Foundation, grant no. AP 252/3-1 and by the "RUDN University program 5-100".

\bibliography{Bibliography_Biharmonic}
\end{document}